\documentclass[10pt]{amsart}

\usepackage{geometry,graphicx,amssymb,amsmath,amsbsy,eucal,amsfonts,mathrsfs,amscd,bm,tcolorbox}

\usepackage{subcaption}
\numberwithin{equation}{section}

\allowdisplaybreaks[3]

\newtheorem{theorem}{Theorem}[section]
\newtheorem{lemma}[theorem]{Lemma}
\newtheorem{assumption}[theorem]{Assumption}

\newtheorem{proposition}[theorem]{Proposition}
\theoremstyle{definition}

\theoremstyle{remark}

\newcommand{\bl}{\bigl\langle}
\newcommand{\br}{\bigr\rangle}

\newcommand{\dive}{{\ensuremath\mathop{\mathrm{div}\,}}}

\newcommand{\bld}[1]{\boldsymbol{#1}}

\newcommand{\bn}{\bld{n}}

\newcommand{\Dt}{\Delta t}
\newcommand{\uu}{\mathsf{u}}
\newcommand{\ww}{\mathsf{w}}
\newcommand{\qq}{\mathsf{q}}
\newcommand{\lla}{\mathsf{l}}
\newcommand{\Ll}{\mathcal{L}}
\newcommand{\Lla}{\Lambda}

\newcommand{\pdt}{\partial_{\Delta t}}
\newcommand{\pt}{\partial_{t}}
\newcommand{\dt}{\Delta t}
\newcommand{\Z}{\mathsf{Z}}
\newcommand{\Ss}{\mathsf{S}}
\newcommand{\J}{\mathsf{J}}

\title[Parabolic/Wave]{Loosely coupled, non-iterative time-splitting scheme based on Robin-Robin coupling: Unified analysis for Parabolic/Parabolic and Parabolic/Hyperbolic problems}

\author{Erik Burman, Rebecca Durst, Miguel Fern\'andez, and Johnny Guzm\'an}

\begin{document}

\maketitle
\begin{abstract}
    We present a loosely coupled, non-iterative time-splitting scheme based on Robin-Robin coupling conditions. We apply a novel unified analysis for this scheme applied to both a Parabolic/Parabolic coupled system and a Parabolic/Hyperbolic coupled system. We show for both systems that the scheme is stable, and the error converges as $\mathcal{O}\big(\dt \sqrt{T +\log{\frac{1}{\dt}}}\big)$, where $\dt$ is the time step.
\end{abstract}

\section{Introduction}
In the physical world, there are a variety of systems in which two distinct media interact, or are coupled, across an interface. For the numerical simulation of such systems, it is highly desirable to develop loosely coupled (explicit) schemes in which the two media are solved sequentially as separate systems by passing information across the interface via defined coupling conditions. The advantage of such schemes is in their efficiency; indeed the alternative strong coupling, in which a fully coupled system must be solved at each time step, may become burdensome when dealing with large and heterogeneous systems. Additionally, loosely coupled schemes present the ability to used optimized, existing solvers for each subsystem.

In order to develop efficient, loosely coupled schemes, the coupling conditions must be chosen appropriately, as poorly chosen coupling conditions can have catastrophic effects for both stability and accuracy. For example, in the case of fluid-structure interaction (FSI) problems the na\"{i}ve choice of coupling conditions has been proven to give rise to unconditional instability under certain physical regimes (see \cite{causin2005added,le-tallec-mouro-01, FWR07}), a phenomenon known as the \textit{added mass effect}.

The objective of this paper is to present and analyze a loosely coupled scheme based on a Robin-Robin type coupling condition applied to two relevant model problems: a parabolic-parabolic interface problem and a parabolic-hyperbolic interface problem. In the continuous case of both of these systems, the coupling conditions at the interface, $\Sigma$, take the form
\begin{equation*}
    \begin{cases}
     \uu = \pt^{k-1}\ww, & \rm{on} \ \Sigma,\\
     \nu_f \nabla \uu \cdot \bn_f + \nu_s \nabla \ww \cdot \bn_s = 0, & \rm{on} \ \Sigma,
    \end{cases}
\end{equation*}
where $\uu$ and $\ww$ are solutions to the individual subsystems, and $k=1,2$ (see below for more details). For the loosely coupled scheme with Robin-Robin type coupling, we discretize appropriately in time and apply the Robin-type coupling conditions
\begin{equation} \label{robinRobin}
\begin{cases}
 \alpha \pdt^{k-1} w^{n+1} + \nu_s\nabla w^{n+1} \cdot \bn_s = \alpha u^n - \nu_f\nabla u^n \cdot \bn_f, & \rm{on} \ \Sigma,\\
 \alpha u^{n+1} + \nu_f\nabla u^{n+1} \cdot \bn_f = \alpha  \pdt^{k-1} w^{n+1} + \nu_f\nabla u^n \cdot \bn_f, & \rm{on} \ \Sigma,
\end{cases}
\end{equation}
where $u^n$ and $w^n$ are the solutions to the time-discrete scheme, and $\alpha>0$ is the so-called Robin parameter.

\subsection{Previous work}
The development of loosely coupled schemes has important consequences, particularly for FSI problems, which may be posed as parabolic-hyperbolic interface problems with an added pressure term. For parabolic-parabolic interface problems, these schemes have been investigated for simulating problems such as heat transport and windblown particulate flow. In \cite{benes2,benes1}, the authors investigate parabolic-parabolic coupling problems with non-matching grids and heterogeneous time-steps, using the transmission condition $u = w$. 

Robin-Robin type coupling was proposed in \cite{lions1990schwarz} as an iterative domain decomposition method for solving the Poisson problem. More recently, it has been applied to partitioned schemes for time dependent interface problems. In \cite{canuto}, the authors investigate a Robin-Robin coupling akin to our method for advection-diffusion problems, with the intention of applying the methods to the simulation of sand transport. In \cite{stokesDarcy}, Cao \textit{et al}  propose a loosely coupled Robin-Robin approach for the Stokes-Darcy problem and prove it to be optimally convergent. 

In the case of FSI problems, Robin-type interface coupling has emerged as a means to decouple an incompressible fluid and a thick-walled structure without encountering the added-mass effect mentioned above. The success of this method was first explored by Burman and Fern\'{a}ndez in \cite{burman2014explicit}, as a variation of their scheme developed in \cite{BURMAN2009766}, which required extra stabilization. Similar methods had previously been explored for the strongly coupled scheme in \cite{badia2008fluid,GNV10,NV08} and in a Robin-based explicit method proposed in \cite{BHS14}. Recently, this approach was shown to be stable, but with only sub-optimal convergence \cite{burman-durst-guzman-19,burman-durst-guzman-fernandez,bukacSeboldt}. 

\subsection{Optimal convergence, unified analysis, and the time-semi-discrete framework.}

\subsubsection{Optimal convergence.}
The main contribution of this paper is the proof of nearly optimal error estimates for the method. More precisely, we provide an error analysis for the parabolic-parabolic and parabolic-hyperbolic systems, resulting in an nearly optimal error estimate of $\mathcal{O}\big(\dt \sqrt{T+\log{\frac{1}{\dt}}}\big)$, without any additional conditions on the time-step or exponential growth. This is the first proof that we are aware of for nearly optimal convergence for a loosely coupled, non-iterative scheme on these types of systems, which is of particular note, since these systems have many similarities to the FSI problem \cite{burman-durst-guzman-19,burman-durst-guzman-fernandez,bukacSeboldt}.


As seen in these papers, previous work on the related FSI system has only succeeded in rigorously proving a sub-optimal error estimate of $\mathcal{O}(\sqrt{T}\sqrt{\dt})$. For the non-iterative Robin-Robin method, the previous numerical test for the FSI problem are not fully conclusive as they did not account for sensitivity to physical parameters or were not sufficiently refined in time. In order to better understand the convergence rates of Robin-Robin coupling methods, we took a step back and herein perform the analysis on the simplified cases of parabolic/parabolic and parabolic/hyperbolic coupling.  We should mention that some papers have developed methods where higher-order convergence has been observed numerically (e.g. second order convergence), such as \cite{burman2014explicit,bukavc2021refactorization,BHS14}, however these methods relied on sub-iterations  or predictor-corrector approach. A rigorous error analysis of higher-order methods has not been carried out, to the best of our knowledge. 

First-order convergence for a Robin-Robin splitting method has been proven for the Stokes-Darcy problem in \cite{stokesDarcy}, which has parallels to the parabolic/parabolic problem investigated in this paper. However, the two problems are distinct in that our analysis deals with an impenetrable boundary and encounters first order terms (specifically $\nabla u \cdot \bn$) on the interface $\Sigma$ that must be treated with care. The authors in \cite{stokesDarcy}, on the other hand, able to avoid a gradient term on the interface.  

In fact, the key to proving our nearly optimal estimates, unlike in the previous analyses of the FSI problem, is  extending the first order terms from the interface $\Sigma$ into the interior using a cut-off function technique (see Section \ref{sectiong3}). We may then apply a summation by parts technique to prove the lifted terms are nearly first order.

\subsubsection{Unified analysis} 
Of particular interest to us in this work is the variety of interface problems for which the Robin-Robin type coupling proves effective. As such, we investigate the coupling's stability and convergence for the parabolic-parabolic and the parabolic-hyperbolic interface problems, as we believe they serve as simplified yet informative models of systems in which Robin-Robin coupling has been applied successfully. For example, in addition to the Stokes-Darcy problem explored in \cite{stokesDarcy}, the parabolic-parabolic interface problem has significant parallels to the methods introduced in \cite{canuto}, and the parabolic-hyperbolic interface problem may be viewed as a simplified version of the standard FSI model without the pressure term and divergence-free condition. However, unlike \cite{canuto}, we do not present a multi-timestep approach, and we are able to assume a homogeneous time step on the whole system.

In order to demonstrate the potentially wide applicability of the Robin-Robin type coupling conditions, we present a novel unified analysis for the parabolic-parabolic interface problem and the parabolic-hyperbolic interface problem, in which the parameter $k=1,2$ serves to distinguish between the two systems. The success of this unified analysis may have interesting implications, especially with respect to its use as a domain decomposition method. For example, it may indicate that the Robin-Robin type coupling does not heavily affect the system beyond the points on the boundary. 

\subsubsection{Time-semi-discrete framework.}
In this paper, we are primarily interested in the convergence in time in order to gain an understanding of the affects of the Robin-Robin coupling conditions. Thus we choose to only discretize in time and analyze the splitting methods in a semi-discrete framework. Hence, no spatial discretization is considered in this paper. 

It should be noted, however, that if standard finite elements are used for the spacial discretization, the stress terms $\nu_f \nabla u \cdot \bn_f$ and $\nabla \pdt^{k-1} w \cdot \bn_s$ in \eqref{robinRobin} will pose an issue with continuity across element boundaries. Therefore, in order to make our work more amenable to a fully discrete framework, we cast our scheme as a Lagrange multiplier method, inspired by the fully discrete scheme in \cite{burman-durst-guzman-fernandez} (see Algorithm 2), which was shown to be a result of a variationally consistent representation of the interface stress involving a lifting operator (see also \cite{le-tallec-mouro-01} and \cite{burman2014explicit}).

The outline of the paper is as follows. In Section \ref{problems}, we present the two interface problems and introduce the unified form and variational formulation that we use in the remainder of the paper. In Section \ref{discrete}, we introduce the time-discrete Robin-Robin coupling method for which we provide a stability analysis in Section \ref{stability}. Then, in Section \ref{error1} we prove the error estimate using a novel extension from the interface $\Sigma$ into the subdomain $\Omega_f$, and in Section \ref{sec:phi} we provide a construction for this extension. Finally, in Section \ref{num} we provide numerical experiments to accompany the results of the previous sections.

\section{The Parabolic-Parabolic and Parabolic-Hyperbolic Problems}\label{problems}
For the two systems that we will consider, let $\Omega$ be a domain that can be decomposed as $\Omega=\Omega_f \cup \Omega_s$ where the common interface $\Sigma=\partial \Omega_f \cap \partial \Omega_s$, like that shown in Figure \ref{fig:fig1}. Below, we define the two interface problems separately before presenting them in their unified form.
\begin{figure}[h]
\includegraphics[width=0.75\linewidth]{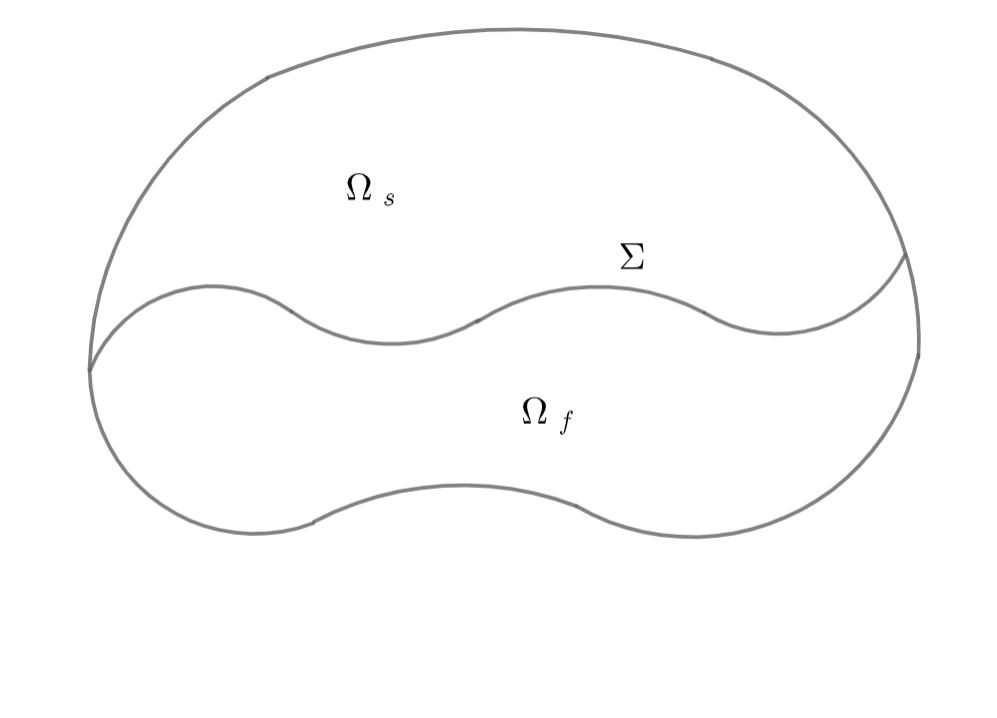}
\caption{An example of the domains $\Omega_s$ and $\Omega_f$ with interface $\Sigma$.\label{fig:fig1}}
\end{figure}

\subsection{The Parabolic-Parabolic problem}
We first consider the standard parabolic problem,
\begin{subequations}\label{noInt}
\begin{alignat}{2}
\pt y-\dive (\nu \nabla y)=&0, \quad && \text{ in } (0, T) \times \Omega, \\
y(0, x)=& y_0(x), \quad  && \text{ in } \Omega, \\
y=&0,  \quad && \text{ on } (0, T) \times \partial \Omega,
\end{alignat}
\end{subequations}
where $\nu$ is a piece-wise constant function given by
\begin{equation}\label{viscosity}
\nu (x) :=
\begin{cases}
\nu_f, & x \in \Omega_f, \\
\nu_s, & x \in \Omega_s.
\end{cases}
\end{equation}

We may then write \eqref{noInt} as an interface problem, where $ \uu := y|_{\Omega_f}$ and $\ww :=y|_{\Omega_s} $, since this is the form we will use for the analysis of the method:
\begin{subequations}\label{PPsplitF}
\begin{alignat}{2}
\pt \uu- \nu_f\Delta \uu=&0, \quad && \text{ in } (0, T) \times \Omega_f, \\
\uu(0, x)=& \uu_0(x), \quad  && \text{ in } \Omega_f, \\
\uu=&0 , \quad && \text{ on }  (0, T) \times \partial \Omega_f \setminus \Sigma,
\end{alignat}
\end{subequations}

\begin{subequations}\label{PPsplitS}
\begin{alignat}{2}
\pt \ww-\nu_s \Delta \ww=&0, \quad && \text{ in }(0, T) \times \Omega_s, \\
\ww(0, x)=& \ww_0(x), \quad  && \text{ in } \Omega_s, \\
\ww=&0,  \quad && \text{ on } (0, T) \times \partial \Omega_s \setminus \Sigma,
\end{alignat}
\end{subequations}
and
\begin{subequations}\label{PPsplitI}
\begin{alignat}{2}
\ww - \uu =& 0, \quad && \text{ in } (0, T) \times \Sigma, \\
\nu_s\nabla \ww \cdot \bn_s + \nu_f \nabla \uu \cdot \bn_f =& 0 \quad && \text{ in } (0, T) \times \Sigma.
\end{alignat}
\end{subequations}

where $\bn_f$ and $\bn_s$ are the outward facing normal vectors for $\Omega_f$ and $\Omega_s$, respectively.

\subsection{The Parabolic-Hyperbolic problem}
Similarly, we consider the standard parabolic and hyperbolic problems in $\Omega_f$ and $\Omega_s$, respectively as an interface problem, where $\nu_f$ and $\nu_s$ are defined as in \eqref{viscosity}.

\begin{subequations}\label{WPsplitF}
\begin{alignat}{2}
\pt\uu- \nu_f\Delta \uu=&0, \quad && \text{ in } (0, T) \times \Omega_f, \\
\uu(0, x)=& \uu_0(x), \quad  && \text{ on } \Omega_f, \\
\uu=&0 , \quad && \text{ on }  (0, T) \times \partial \Omega_f \setminus \Sigma,
\end{alignat}
\end{subequations}

\begin{subequations}\label{WPsplitS}
\begin{alignat}{2}
\pt\qq-\nu_s \Delta \ww=&0, \quad && \text{ in } (0, T) \times \Omega_s, \\
\qq =& \pt\ww, \quad && \text{ in } (0, T) \times \Omega_s, \\
\qq(0,x) =& \qq_0(x), \quad && \text{ on } \Omega_s, \\
\ww(0, x)=& \ww_0(x), \quad  && \text{ on } \Omega_s, \\
\ww=&0,  \quad && \text{ on }  (0, T) \times \partial \Omega_s \setminus \Sigma,
\end{alignat}
\end{subequations}
and
\begin{subequations}\label{WPsplitI}
\begin{alignat}{2}
\qq - \uu =& 0, \quad && \text{ in } (0, T) \times \Sigma, \\
\nu_s\nabla \ww \cdot \bn_s + \nu_f \nabla \uu \cdot \bn_f =& 0, \quad && \text{ in } (0, T) \times \Sigma.
\end{alignat}
\end{subequations}
where $\bn_f$ and $\bn_s$ are defined as before.

\subsection{The generalized system}
To avoid redundancy in the analyses, we will work instead with a unified version that can represent either the parabolic-parabolic system described in \eqref{noInt}-\eqref{PPsplitI}, or the parabolic-hyperbolic system described in \eqref{WPsplitF}-\eqref{WPsplitI}, depending on the choice of integer, $k$. More specifically, when $k=1$, we recover \eqref{PPsplitF}-\eqref{PPsplitI}, and when $k=2$, we recover \eqref{WPsplitF}-\eqref{WPsplitI}.  

Then our generalized system may be written for $k=1,2$,
\begin{subequations}\label{GsplitF}
\begin{alignat}{2}
\pt\uu- \nu_f\Delta \uu=&0,\quad && \text{ in } (0, T) \times \Omega_f, \\
\uu(0, x)=& \uu_0(x), \quad  && \text{ on } \Omega_f, \\
\uu=&0 , \quad && \text{ on }  [0, T] \times \partial \Omega_f \setminus \Sigma,
\end{alignat}
\end{subequations}
\begin{subequations}\label{GsplitS}
\begin{alignat}{2}
\pt\qq-\nu_s \Delta \ww=&0, \quad && \text{ in } (0, T) \times \Omega_s, \\
\qq =& \pt^{k-1}\ww, \quad && \text{ in } (0, T) \times \Omega_s, \\
\ww(0, x)=& \ww_0(x), \quad  && \text{ on } \Omega_s, \\
\qq(0, x)=& \qq_0(x), \quad  && \text{ on } \Omega_s, \label{qq0} \\
\ww=&0,  \quad && \text{ on }  (0, T) \times \partial \Omega_s \setminus \Sigma,
\end{alignat}
\end{subequations}
and
\begin{subequations}\label{GsplitI}
\begin{alignat}{2}
\qq - \uu =& 0, \quad && \text{ in } (0, T) \times \Sigma, \\
\nu_s\nabla \ww \cdot \bn_s + \nu_f \nabla \uu \cdot \bn_f =& 0, \quad && \text{ in } (0, T) \times \Sigma.
\end{alignat}
\end{subequations}
When $k=1$ we assume that $\qq_0=\ww_0$ and, hence, \eqref{qq0} is redundant in that case.  

\subsection{Variational form of the generalized system}
Let $(\cdot, \cdot)_i$ be the $L^2$-inner product on $\Omega_i$ for $i=f, s$. Moreover, let $\bl \cdot, \cdot \br$ be the $L^2$-inner product on $\Sigma$.  Let $N>0$ be an integer, and define $\Dt = \frac{T}{N}$, and let $\uu^{n} = \uu(t_n, \cdot)$, where $t_n = n\Dt$ for $n \in \{0, 1, 2, \dots, N\}$. We consider the spaces
\begin{alignat*}{1}
 V_f=&\{ v \in H^1(\Omega_f): v=0 \text{ on }  \partial \Omega_f \backslash \Sigma \},  \\
 V_s= &\{ v \in H^1(\Omega_s): v=0 \text{ on  } \partial \Omega_s \backslash \Sigma \},  \\
V_g= & L^2(\Sigma).
\end{alignat*}

 Define  $\lla^{n+1}= \nu_f \nabla \uu^{n+1} \cdot \bn_f$ and  assuming that $\lla^{n+1} \in L^2(\Sigma)$ for all $n$, the solutions to \eqref{GsplitF}-\eqref{GsplitI} at time $t_{n+1}$ also satisfy the following problem:\\

Find $\uu^{n+1} \in V_f$, $\qq^{n+1},\ww^{n+1}  \in V_s$, and $\lla^{n+1} \in V_g $ such that
\begin{subequations}\label{var}
\begin{alignat}{2}
(\partial_t \qq^{n+1},z )_s+ \nu_s (\nabla \ww^{n+1}, \nabla z)_s+ \bl \lla^{n+1}, z\br=&0 , \quad && z \in V_s\label{var1} \\
(\partial_t \uu^{n+1}, v)_f+\nu_f (\nabla \uu^{n+1}, \nabla v)_f- \bl \lla^{n+1}, v\br=&0, \quad &&v \in V_f \label{var2}\\
(\qq^{n+1} - \partial_t^{k-1} \ww^{n+1}, r)_s =&0, \quad &&r \in V_s \label{var3} \\
\bl \qq^{n+1}-\uu^{n+1}, \mu \br=&0, \quad && \mu \in V_g \label{var4}
\end{alignat}
\end{subequations}

\section{Robin-Robin coupling: Time discrete method} \label{discrete}

For the time discrete Robin-Robin method, we use a backward-Euler method for the parabolic system on $\Omega_f$. When $k=1$, we also use a backward-Euler method for the resulting parabolic system on $\Omega_s$, and when $k=2$, we use a Newmark method for the resulting hyperbolic system on $\Omega_s$. Thus, we define the discrete derivative in time
\begin{alignat*}{1}
\pdt v^{n+1}=\frac{v^{n+1}-v^n}{\dt},
\end{alignat*}
and the discrete average in time
\begin{alignat*}{1}
v^{n+1/2}=\frac{v^{n+1}+v^n}{2}.
\end{alignat*}
We will also use the notation
\begin{alignat*}{1}
\pdt^j v &=
\begin{cases}
 v, & j= 0, \\
 \pdt v, & j=1, \\
\end{cases}
\end{alignat*}

The Robin-Robin splitting method is as follows. \\

Find $q^{n+1}, w^{n+1}  \in V_s$, $u^{n+1} \in V_f$, and $\lambda^{n+1} \in V_g $ such that, for $k=1,2$ and $n\geq 0$,
\begin{subequations}\label{s}
\begin{alignat}{2}
(\pdt q^{n+1}, z)_s+\nu_s(\nabla w^{n+1/k}, \nabla z)_s + \alpha \bl \pdt^{k-1} w^{n+1} - u^n, z \br + \bl \lambda^n, z\br=&0 , \quad && z \in V_s, \label{s1}\\
(q^{n+1/k} - \pdt^{k-1} w^{n+1}, r)_s =&0 , \quad && r \in V_s, \label{s4}\\
(\pdt u^{n+1}, v)_f+\nu_f(\nabla u^{n+1}, \nabla v)_f- \bl \lambda^{n+1}, v\br=&0, \quad &&v \in V_f,  \label{s2}\\
\bl \alpha(u^{n+1}-\pdt^{k-1} w^{n+1})+(\lambda^{n+1}-\lambda^n), \mu \br=&0, \quad && \mu \in V_g,  \label{s3}
\end{alignat}
\end{subequations}
We note here that the two sub-problems in \eqref{s} are well-posed. 

As a consequence of \eqref{s3} we have 
\begin{equation}\label{con}
 \alpha(u^{n+1}-\pdt^{k-1} w^{n+1})=\lambda^n-\lambda^{n+1}, \quad \text{ on } \Sigma.
\end{equation}
Similarly, from \eqref{s4} we can show
\begin{equation}\label{strong}
q^{n+1/k} = \pdt^{k-1}w^{n+1}, \quad \text{ on } \Omega_s, 
\end{equation}
and thus
\begin{equation}\label{strong2}
 \alpha(u^{n+1}-q^{n+1/k})=\lambda^n-\lambda^{n+1},  \quad \text{ on } \Sigma.
\end{equation}

In practice, this splitting method would be implemented sequentially. We assume that information from the previous time-step is known, specifically $q^{n}$, $w^n$, $u^n$, and $\lambda^n$, so $q^{n+1}$ and $w^{n+1}$ may be solved using \eqref{s1} and \eqref{s4}. This solution is then applied as data in \eqref{s3}, which serves as an interface coupling equation. We may then solve for $u^{n+1}$ and $\lambda^{n+1}$ using \eqref{s2} and \eqref{s3}. 

In this way, the Robin-Robin splitting method is a loosely coupled scheme for the numerical time semi-discrete approximation of \eqref{GsplitF}-\eqref{GsplitI}.

\section{Stability} \label{stability}
In this section we prove stability of the method \eqref{s}. We will use the following identity where $(\cdot, \cdot)$ is an inner product and $\| \cdot \|$ is the corresponding norm. 
\begin{equation}\label{id}
(\phi-\psi, \theta)=\frac{1}{2}(\|\phi\|^2-\|\psi\|^2+ \|\theta-\psi\|^2 - \|\theta - \phi\|^2).
\end{equation}

We need to define the following quantities:
\begin{alignat*}{1}
Z^{n+1}:= & \frac{1}{2}\|q^{n+1}\|_{L^2(\Omega_s)}^2+ \frac{1}{2} \|u^{n+1}\|_{L^2(\Omega_f)}^2  + \frac{(k-1)\nu_s}{2} \|\nabla w^{n+1}\|_{L^2(\Omega_s)}^2 +\frac{\dt \alpha}{2} \|u^{n+1}\|_{L^2(\Sigma)}^2 \\
&+ \frac{\dt}{2\alpha} \|\lambda^{n+1}\|_{L^2(\Sigma)}^2 \\ 
S^{n+1}:= & \nu_f \dt \|\nabla u^{n+1}\|_{L^2(\Omega_f)}^2+  (2-k)\nu_s\dt \|\nabla w^{n+1}\|_{L^2(\Omega_s)}^2 +  \frac{(2-k)}{2}\|q^{n+1}-q^n\|_{L^2(\Omega_s)}^2\\
&+ \frac{1}{2} \|u^{n+1}-u^n\|_{L^2(\Omega_f)}^2  + \frac{\dt \alpha}{2}\|q^{n+1/k}-u^n\|_{L^2(\Sigma)}^2.
\end{alignat*}

\begin{lemma}
It holds,
\begin{equation*}
Z^N+\sum_{n=0}^{N-1} S^{n+1} =Z^0. 
\end{equation*}

\end{lemma}

\begin{proof}
 If we set $z=\dt q^{n+1/k}$ and $v=\dt u^{n+1}$ in \eqref{s1} and \eqref{s2}, respectively, and apply \eqref{strong}, we get, for $n\geq 0$, 
\begin{alignat}{1}
& \frac{1}{2}\|q^{n+1}\|_{L^2(\Omega_s)}^2+ \frac{1}{2} \|u^{n+1}\|_{L^2(\Omega_f)}^2 +  \frac{(2-k)}{2}\|q^{n+1}-q^n\|_{L^2(\Omega_s)}^2+ \frac{1}{2} \|u^{n+1}-u^n\|_{L^2(\Omega_f)}^2 \nonumber  \\
 & + \frac{(k-1)\nu_s}{2} \|\nabla w^{n+1}\|_{L^2(\Omega_s)}^2+ (2-k)\nu_s\dt \|\nabla w^{n+1}\|_{L^2(\Omega_s)}^2+  \nu_f\dt \|\nabla u^{n+1}\|_{L^2(\Omega_f)}^2 \nonumber \\
  & =  \frac{1}{2}\|q^{n}\|_{L^2(\Omega_s)}^2+ \frac{1}{2} \|u^{n}\|_{L^2(\Omega_f)}^2+ \frac{(k-1) \nu_s}{2} \|\nabla w^{n}\|_{L^2(\Omega_s)}^2 +\dt J^{n+1}. \label{aux213}
\end{alignat}
where
\begin{alignat*}{1}
J^{n+1}:=& -  \alpha \bl \pdt^{k-1}w^{n+1} - u^n, q^{n+1/k} \br-   \bl \lambda^n, q^{n+1/k}\br+ \bl \lambda^{n+1}, u^{n+1}\br.
\end{alignat*}
Using \eqref{strong} and \eqref{strong2} and after some manipulations we have
\begin{alignat*}{1}
J^{n+1}=& \alpha \bl u^n-u^{n+1}, u^{n+1} \br+ \frac{1}{\alpha} \bl  \lambda^n-\lambda^{n+1}, \lambda^{n+1} \br- \bl u^n-u^{n+1}, \lambda^n-\lambda^{n+1} \br. 
\end{alignat*}
If we use we \eqref{id} on the first two terms we obtain
\begin{alignat*}{1}
J^{n+1} =& \frac{\alpha}{2}(\|u^n\|_{L^2(\Sigma)}^2 - \|u^{n+1}\|_{L^2(\Sigma)}^2)  + \frac{1}{2\alpha}(\|\lambda^n\|_{L^2(\Sigma)}^2- \|\lambda^{n+1}\|_{L^2(\Sigma)}^2) \\
&-   \frac{\alpha}{2}\|u^n-u^{n+1}\|_{L^2(\Sigma)}^2 -\frac{1}{2\alpha}\|\lambda^n-\lambda^{n+1}\|_{L^2(\Sigma)}^2  - \bl u^n-u^{n+1}, \lambda^n-\lambda^{n+1} \br.
\end{alignat*}
Thus, a simple identity applied to the last three terms gives
\begin{alignat}{1}
J^{n+1} =& \frac{\alpha}{2}(\|u^n\|_{L^2(\Sigma)}^2 - \|u^{n+1}\|_{L^2(\Sigma)}^2)  + \frac{1}{2\alpha}(\|\lambda^n\|_{L^2(\Sigma)}^2 - \|\lambda^{n+1}\|_{L^2(\Sigma)}^2) \nonumber \\
&- \frac{\alpha}{2}\|(u^n-u^{n+1})+ \frac{1}{\alpha}(\lambda^n-\lambda^{n+1}) \|_{L^2(\Sigma)}^2. \label{aux561}
\end{alignat}
Finally, using \eqref{strong2} we have $u^n-u^{n+1}+ \frac{1}{\alpha}(\lambda^n-\lambda^{n+1})=u^n -q^{n+1/k}$ and so  plugging this result into  \eqref{aux213}, we arrive at  
\begin{equation*}
    Z^{n+1} + S^{n+1} = Z^n.
\end{equation*}
Taking the sum yields the final result.
\end{proof}

\section{Error Analysis}\label{error1}
In this section we prove the main error estimate. We start by deriving the error equations. We denote the error variables
\begin{alignat*}{2}
U^n&=\uu^n-u^n, \quad  Q^n &&= \qq^n-q^n, \\ W^n&=\ww^n-w^n, \quad  \Lla^n &&=\lla^n-\lambda^n.
\end{alignat*}
We assume that we chose the initial conditions of the splitting method to be exactly the initial conditions of the coupled problem and so these quantities vanish when $n=0$. Using \eqref{s1}-\eqref{s3} and \eqref{var1}-\eqref{var4}, we recover the error equations. We thus have
\begin{subequations}\label{errorRR}
\begin{alignat}{1}\label{errorRR1}
(\pdt Q^{n+1}, z)_s+ \nu_s(\nabla W^{n+1/k}, \nabla z)_s+\alpha \bl (\pdt^{k-1} W^{n+1} - U^n), z\br+ \bl \Lla^n, z \br =& L_1(z) - L_4(z) , \\
(Q^{n+1/k} -\pdt^{k-1} W^{n+1} , r)_s =& L_2(r), \label{errorRR12} \\
 (\pdt U^{n+1}, v)_f+ \nu_f (\nabla U^{n+1}, \nabla v)_f -  \bl  \Lla^{n+1},v\br
= & L_3(v), \label{errorRR2}\\
 \alpha \bl U^{n+1} - \pdt^{k-1} W^{n+1},\mu\br +\bl \Lla^{n+1} - \Lla^{n}, \mu\br = & L_4(\mu).\label{errorRR3}
\end{alignat}
\end{subequations}
where 
\begin{alignat*}{1}
L_1(z) :=&\frac{(k-1)}{2}\bl g_2^{n+1}, z \br + \alpha \bl g_1^{n+1}, z \br - ( h_1^{n+1}, z)_s,\\
L_2(r) :=& (h_3^{n+1}, r)_s,\\
L_3(v):=&-(h_2^{n+1},v)_f,\\
L_4(\mu):=& \alpha \bl h_4^{n+1}, \mu \br + \bl g_2^{n+1}, \mu \br,
\end{alignat*}
and
\begin{alignat*}{2}
h_1^{n+1} &:= \pt \qq^{n+1/k} - \pdt \qq^{n+1}, \quad && g_1^{n+1} := \uu^{n+1}-\uu^n, \\
h_2^{n+1} &:= \pt \uu^{n+1} - \pdt \uu^{n+1}, \quad && g_2^{n+1} := \lla^{n+1} - \lla^n, \\
h_3^{n+1} &:= \pt^{k-1} \ww^{n+1/k} - \pdt^{k-1} \ww^{n+1}, \\
h_4^{n+1} &:= \pt^{k-1}\ww^{n+1} - \pdt^{k-1} \ww^{n+1}. 
\end{alignat*}
We note that to determine \eqref{errorRR1}, we used the following form of \eqref{s1}
\begin{equation*}
(\pt \qq^{n+1/k},z)_S + \nu_s (\nabla \ww^{n+1/k},\nabla z)_s + \bl \lla^{n+1/k},z\br =0, \quad \forall z \in V_s.
\end{equation*}

We also note that as a direct consequence of \eqref{errorRR12}, we may write
\begin{equation}\label{errorStrong}
Q^{n+1/k} = \pdt^{k-1} W^{n+1} + h_3^{n+1}, \quad \text{ on } \Omega_s.
\end{equation}

It also follows from \eqref{errorRR3} that
\begin{equation}\label{errorStrong2}
\alpha(U^{n+1} - \pdt^{k-1} W^{n+1} ) = \Lambda^n - \Lambda^{n+1} +\alpha h_4^{n+1}+ g_2^{n+1}, \quad \text{ on } \Sigma.
\end{equation}
Thus, using the fact that $\pt^{k-1}\ww^n = \uu^n$ on $\Sigma$, \eqref{errorStrong} and \eqref{errorStrong2} combine so that we may write
\begin{equation}\label{errorStrong3}
\alpha \big( U^{n+1} - Q^{n+1/k} \big) =  \Lla^{n} - \Lla^{n+1} +g_3^{n+1},   \quad \text{ on } \Sigma.
\end{equation}
where 
\begin{equation*}
g_3^{n+1}:= \alpha \frac{(k-1)}{2}  g_1^{n+1}+ g_2^{n+1}.
\end{equation*}

We define the quantities that will allow to perform the error analysis. 
\begin{alignat*}{1}
\Z^{n+1}:= & \frac{1}{2}\|Q^{n+1}\|_{L^2(\Omega_s)}^2+ \frac{1}{2} \|U^{n+1}\|_{L^2(\Omega_f)}^2 +\frac{(k-1)\nu_s}{2} \|\nabla W^{n+1}\|_{L^2(\Omega_s)} +\frac{\dt \alpha}{2} \|U^{n+1}\|_{L^2(\Sigma)}^2 \\
&+ \frac{\dt}{2\alpha} \|\Lambda^{n+1}\|_{L^2(\Sigma)}^2, \\ 
\Ss^{n+1}:= & \nu_f \dt \|\nabla U^{n+1}\|_{L^2(\Omega_f)}^2+  (2-k)\nu_s\dt \|\nabla W^{n+1}\|_{L^2(\Omega_s)}^2 +  \frac{(2-k)}{2}\|Q^{n+1}-Q^n\|_{L^2(\Omega_s)}^2\\
&+ \frac{1}{2} \|U^{n+1}-U^n\|_{L^2(\Omega_f)}^2  + \frac{\dt \alpha}{2}\|U^n-Q^{n+1/k}\|_{L^2(\Sigma)}^2+ \frac{1}{2\alpha} \| g_3^{n+1}\|_{L^2(\Sigma)}^2.
\end{alignat*}

\subsection{Preliminary Estimates}
Before proving our main result, in this subsection, we will prove some preliminary results. To this end, we start with the following lemma. 
\begin{lemma}\label{errorLemma1}It holds, 
\begin{alignat*}{1}
\Z^{n+1}+\Ss^{n+1}=\Z^n + \dt F^{n+1}  +\frac{\dt}{\alpha} \bl g_3^{n+1}, \Lambda^{n+1} \br,
\end{alignat*}
where
\begin{alignat*}{1}
F^{n+1}:= & -(h_1^{n+1}, Q^{n+1/k})_s-(h_2^{n+1}, U^{n+1})_f- \nu_s(\nabla W^{n+1/k}, \nabla h_3^{n+1})_s  \\
&+ \bl g_4^{n+1}, Q^{n+1/k} \br + \bl U^{n+1}-U^n, g_3^{n+1} \br - \bl g_3^{n+1},Q^{n+1/k}-U^n\br,
\end{alignat*}
and 
\begin{equation*}
g_4^{n+1}:= \alpha g_1^{n+1}+\frac{k-1}{2} g_2^{n+1}.
\end{equation*}

\end{lemma}
\begin{proof}

To begin, we set $z= \dt Q^{n+1/k}$ in \eqref{errorRR1} and $v = \dt U^{n+1}$ in \eqref{errorRR2}  and use \eqref{errorStrong} to get
\begin{alignat}{1}
& \frac{1}{2}\|Q^{n+1}\|_{L^2(\Omega_s)}^2+ \frac{1}{2} \|U^{n+1}\|_{L^2(\Omega_f)}^2 +  \frac{(2-k)}{2}\|Q^{n+1}-Q^n\|_{L^2(\Omega_s)}^2+ \frac{1}{2} \|U^{n+1}-U^n\|_{L^2(\Omega_f)}^2 \nonumber  \\
 & + \frac{(k-1)\nu_s}{2} \|\nabla W^{n+1}\|_{L^2(\Omega_s)} + (2-k)\nu_s\dt \|\nabla W^{n+1}\|_{L^2(\Omega_s)}^2+  \nu_f\dt \|\nabla U^{n+1}\|_{L^2(\Omega_f)}^2 \nonumber \\
  & =  \frac{1}{2}\|Q^{n}\|_{L^2(\Omega_s)}^2+ \frac{1}{2} \|U^{n}\|_{L^2(\Omega_f)}^2+ \frac{(k-1) \nu_s}{2} \|\nabla W^{n}\|_{L^2(\Omega_s)}^2 +\dt \J^{n+1}. \label{erroraux213}
\end{alignat}
where
\begin{alignat*}{1}
\J^{n+1}:=& -  \alpha \bl \pdt^{k-1}W^{n+1} - U^n, Q^{n+1/k} \br- \bl  \Lambda^n, Q^{n+1/k}\br+ \bl \Lambda^{n+1}, U^{n+1}\br -\nu_s(\nabla W^{n+1/k}, \nabla h_3^{n+1})_s\\
 &+L_1(Q^{n+1/k}) - L_4(Q^{n+1/k}) +L_3(U^{n+1}).
\end{alignat*}
We note that the term $\nu_s(\nabla W^{n+1/k}, \nabla h_3^{n+1})_s$ appears when we apply \eqref{errorStrong} such that
\begin{alignat*}{1}
\nu_s(\nabla W^{n+1/k}, \dt Q^{n+1/2})_s & = \nu_s(\nabla W^{n+1/k}, W^{n+1} - W^n + \dt h_3^{n+1})_s.
\end{alignat*}
Manipulating the first three terms and using \eqref{errorStrong} and \eqref{errorStrong3} we obtain
\begin{alignat*}{1}
& -  \alpha \bl \pdt^{k-1}W^{n+1} - U^n, Q^{n+1/k} \br-   \bl \Lambda^n, Q^{n+1/k}\br+ \bl \Lambda^{n+1}, U^{n+1}\br  \\
=& \mathbb{J}^{n+1} +  \alpha \bl h_3^{n+1}, Q^{n+1/k} \br+ \frac{1}{\alpha} \bl g_3^{n+1}, \Lambda^{n+1} \br + \bl g_3^{n+1}, Q^{n+1/k} \br- \bl U^n-U^{n+1}, g_3^{n+1} \br,
\end{alignat*}
where 
\begin{equation*}
 \mathbb{J}^{n+1} := \alpha \bl U^n-U^{n+1}, U^{n+1} \br+ \frac{1}{\alpha} \bl  \Lambda^n-\Lambda^{n+1}, \Lambda^{n+1} \br- \bl U^n-U^{n+1}, \Lambda^n-\Lambda^{n+1} \br.
\end{equation*}
Finally, as we did in the stability analysis  (see \eqref{aux561}) we get
\begin{alignat*}{1}
 \mathbb{J}^{n+1}= & \frac{\alpha}{2}(\|U^n\|_{L^2(\Sigma)}^2 - \|U^{n+1}\|_{L^2(\Sigma)}^2)  + \frac{1}{2\alpha}(\|\Lambda^n\|_{L^2(\Sigma)}^2 - \|\Lambda^{n+1}\|_{L^2(\Sigma)}^2) \\
&- \frac{\alpha}{2}\|(U^n-U^{n+1})+ \frac{1}{\alpha}(\Lambda^n-\Lambda^{n+1}) \|_{L^2(\Sigma)}^2.
\end{alignat*}
Using \eqref{errorStrong3}  we can re-write the last term 
\begin{alignat*}{1}
&\|(U^n-U^{n+1})+ \frac{1}{\alpha}(\Lambda^n-\Lambda^{n+1}) \|_{L^2(\Sigma)}^2\\
=&\| U^n- Q^{n+1/k}- \frac{1}{\alpha} g_3^{n+1}\|_{L^2(\Sigma)}^2 \\
=& \| U^n- Q^{n+1/k}\|_{L^2(\Sigma)}^2+ \frac{1}{\alpha^2} \| g_3^{n+1}\|_{L^2(\Sigma)}^2+\frac{2}{\alpha} \bl Q^{n+1/k}-U^n, g_3^{n+1} \br.  
\end{alignat*}
Plugging this back into $\J^{n+1}$ we arrive at
\begin{alignat*}{1}
\J^{n+1}=&  \frac{\alpha}{2}\big(\|U^n\|_{L^2(\Sigma)}^2 - \|U^{n+1}\|_{L^2(\Sigma)}^2\big) +\frac{1}{2\alpha}\big(\|\Lambda^n\|_{L^2(\Sigma)}^2 - \|\Lambda^{n+1}\|_{L^2(\Sigma)}^2\big)-  \frac{\alpha}{2} \| U^n- Q^{n+1/k}\|_{L^2(\Sigma)}^2\\
&-\frac{1}{2\alpha} \| g_3^{n+1}\|_{L^2(\Sigma)}^2  - \bl g_3^{n+1}, Q^{n+1/k}-U^n  \br -\nu_s(\nabla W^{n+1/k}, \nabla h_3^{n+1})_s\\
&+L_1(Q^{n+1/k}) - L_4(Q^{n+1/k}) +L_3(U^{n+1}) \\
& +  \alpha \bl h_3^{n+1}, Q^{n+1/k} \br+ \frac{1}{\alpha} \bl g_3^{n+1}, \Lambda^{n+1} \br + \bl g_3^{n+1}, Q^{n+1/k} \br- \bl g_3^{n+1} , U^n-U^{n+1} \br.
\end{alignat*}

We now note, using \eqref{errorStrong}, \eqref{errorStrong2}, and \eqref{errorStrong3}, that
\begin{alignat*}{1}
-L_4(Q^{n+1/k}) + \alpha \bl h_3^{n+1}, Q^{n+1/k}\br & = -\bl \alpha h_4^{n+1} + g_2^{n+1} - \alpha h_3^{n+1}, Q^{n+1/k}\br \\
&= -\bl \alpha(U^{n+1} - Q^{n+1/k}) - (\Lambda^n - \Lambda^{n+1}), Q^{n+1/k}\br \\
&= -\bl g_3^{n+1}, Q^{n+1/k}\br.
\end{alignat*}

Thus we have
\begin{alignat*}{1}
\J^{n+1}=& \frac{\alpha}{2}\big(\|U^n\|_{L^2(\Sigma)}^2 - \|U^{n+1}\|_{L^2(\Sigma)}^2\big) +\frac{1}{2\alpha}\big(\|\Lambda^n\|_{L^2(\Sigma)}^2 - \|\Lambda^{n+1}\|_{L^2(\Sigma)}^2\big) -  \frac{\alpha}{2} \| U^n- Q^{n+1/k}\|_{L^2(\Sigma)}^2\\
&-\frac{1}{2\alpha} \| g_3^{n+1}\|_{L^2(\Sigma)}^2 + F^{n+1} +\frac{1}{\alpha} \bl g_3^{n+1}, \Lambda^{n+1} \br.
\end{alignat*}
If we plug in these results to \eqref{erroraux213} we arrive at the identity. 
\end{proof}

As the reader can infer we singled out the term  $\frac{\dt}{\alpha} \bl g_3^{n+1}, \Lambda^{n+1} \br$ as this one needs special care. As we will see, the terms appearing in $F^{n+1}$ can be bounded easily and they will contribute $O(\dt)$ which is optimal. In an analogous FSI problem the term corresponding  to $\frac{\dt}{\alpha} \bl g_3^{n+1}, \Lambda^{n+1} \br$ was bounded in \cite{burman-durst-guzman-fernandez} and lead to a sub-optimal error estimate $O(\sqrt{\dt})$. One of the main contributions of this paper is to give an alternative bound of this term that will lead to a nearly first order estimate.

\subsubsection{Estimate for $F^{n+1}$}
We will need a Poincare-Friedrichs type inequality and a trace inequality.
\begin{proposition}
There exists constants $C_P$ and $C_{\text{tr}}$ such that 

\begin{equation}\label{poincare}
\|v\|_{L^2(\Omega_f)} \le C_P \|\nabla v\|_{L^2(\Omega)} \quad \forall v \in V_s,
\end{equation}
and
\begin{equation}\label{trace}
\|v\|_{L^2(\Sigma)} \le C_{\text{tr}} \|\nabla v\|_{L^2(\Omega)} \quad \forall v \in V_s.
\end{equation}
\end{proposition}


We now estimate  the sum of $F^{n+1}$.
\begin{lemma}\label{Festimate}
Let $1 \le M \le N$, then
\begin{equation}
 \dt \sum_{n=0}^{M-1} F^{n+1} \le   \frac{1}{4} \max_{1\le n \le M}  \Z^n+ \frac{1}{4} \sum_{n=0}^{M-1} \Ss^{n+1}+ C D(M),
\end{equation}
where
\begin{alignat*}{1}
D(M):=& \dt T \sum_{n=0}^{M-1} \Big(\|h_1^{n+1}\|_{L^2(\Omega_s)}^2 +\|h_2^{n+1}\|_{L^2(\Omega_f)}^2+(k-1)\nu_s \|\nabla h_3^{n+1}\|_{L^2(\Omega_s)}^2 \Big)  \\
&+ \dt  \sum_{n=0}^{M-1}  \bigg(\frac{1}{\alpha}+\frac{C_\text{tr}^2}{\nu_f}\bigg)\Bigg(\|g_4^{n+1}\|_{L^2(\Sigma)}^2+  \|g_3^{n+1}\|_{L^2(\Sigma)}^2 \Bigg). 
\end{alignat*}
\end{lemma}

\begin{proof}
We see that
 \begin{alignat*}{1}
& -\dt  \sum_{n=0}^{M-1}  \Big( (h_1^{n+1}, Q^{n+1/k})_s+(h_2^{n+1}, U^{n+1})_f+ \nu_s(\nabla W^{n+1/k}, \nabla h_3^{n+1})_s \Big) \\
\le &   \frac{\dt}{8T} \sum_{n=0}^{M-1} \Big( \|   Q^{n+1/k}\|_{L^2(\Omega_s)}^2 + \|   U^{n+1}\|_{L^2(\Omega_f)}^2+  (k-1) \nu_s \|   \nabla W^{n+1/k}\|_{L^2(\Omega_s)}^2\Big)  \\
& + C  \dt T \sum_{n=0}^{M-1} \Big(\|h_1^{n+1}\|_{L^2(\Omega_s)}^2 +\|h_2^{n+1}\|_{L^2(\Omega_f)}^2+(k-1)\nu_s \|\nabla h_3^{n+1}\|_{L^2(\Omega_s)}^2 \Big) \\
\le &   \frac{1}{4} \max_{1\le n \le M}  \Z^n+  C  \dt T \sum_{n=0}^{M-1} \Big(\|h_1^{n+1}\|_{L^2(\Omega_s)}^2 +\|h_2^{n+1}\|_{L^2(\Omega_f)}^2+(k-1)\nu_s \|\nabla h_3^{n+1}\|_{L^2(\Omega_s)}^2 \Big). 
\end{alignat*}

Here, we used the fact that $h_3^{n+1} = 0$ when $k=1$, so the term $\nu_s(\nabla W^{n+1/k},\nabla h_3^{n+1})_s$ is only present when $k=2$, as indicated by the factor $(k-1)$.
 
We also have the bound
\begin{alignat*}{1}
& \dt  \sum_{n=0}^{M-1} \Big( \bl g_4^{n+1}, Q^{n+1/k} \br+ \bl g_3^{n+1}, U^{n+1}-U^n  \br - \bl g_3^{n+1}, Q^{n+1/k}-U^n\br \Big) \\
=& \dt  \sum_{n=0}^{M-1}  \Big(\bl g_4^{n+1}, Q^{n+1/k}-U^n \br +  \bl g_4^{n+1}, U^n \br \\
&\hspace{1.5cm}+ \bl g_3^{n+1}, U^{n+1}-U^n  \br - \bl g_3^{n+1}, Q^{n+1/k}-U^n\br  \Big) \\
\le &  C \dt  \sum_{n=0}^{M-1}  \bigg(\frac{1}{\alpha}+\frac{C_\text{tr}^2}{\nu_f}\bigg) \Bigg( \|g_4^{n+1}\|_{L^2(\Sigma)}^2+ \|g_3^{n+1}\|_{L^2(\Sigma)}^2 \Bigg) \\
&\hspace{1.5cm}+ \dt    \sum_{n=0}^{M-1}  \Big(\frac{\alpha}{8}\| Q^{n+1/k}-U^n \|_{L^2(\Sigma)}^2+  \frac{\nu_f}{8} \|\nabla U^{n+1}\|_{L^2(\Omega_f)}^2 \Big) \\
\leq&  \frac{1}{4}\sum_{n=0}^{M-1} \Ss^{n+1} +  C\dt \sum_{n=0}^{M-1}\bigg(\frac{1}{\alpha}+\frac{C_\text{tr}^2}{\nu_f}\bigg)\Bigg( \|g_4^{n+1}\|_{L^2(\Sigma)}^2+\|g_3^{n+1}\|_{L^2(\Sigma)}^2  \Bigg) ,
\end{alignat*}
where we used \eqref{trace}. Combining these two inequalities proves the result. 
\end{proof}

\subsubsection{Estimate for   $\frac{\dt}{\alpha} \bl g_3^{n+1}, \Lambda^{n+1} \br$ using a lifting-residual argument}\label{sectiong3}
In this section we show how to estimate the term 
\begin{equation*}
\frac{\dt}{\alpha} \sum_{n=0}^{M-1} \bl g_3^{n+1}, \Lambda^{n+1} \br.
\end{equation*}
The idea is to use \eqref{errorRR2}, however, in order to do so we need to extend $g_3^{n+1}$ into $\Omega_f$ in such a way that the extension belongs to $V_f$. In particular, the extension needs to vanish on $\partial \Omega_f \backslash \Sigma$. This will not be possible in general and, therfore, we will need to utilize a cut-off function technique.

We will make  two assumptions. The first is that the normal $\bn$ can be extended from $\Sigma$ to $\Omega_f$ in such that the extension has a bounded gradient. For example, this can be done if the interface $\Sigma$ is smooth. Furthermore, if $\Sigma$ is a straight line, this extension is trivial because $\bn$ will be constant.

 \begin{assumption}\label{ntilde}
 There exists $\tilde{\bn} \in [W^{1,\infty}(\Omega_f)]^2$ such that  $\tilde{\bn}|_{\Sigma} = \bn$. 
 \end{assumption}

The second assumption regards the existence of a cut-off function, dependent on $\dt$, that is one on most of $\Sigma$ such that the gradient can be controlled appropriately in the $L^2$ norm.

\begin{assumption}\label{def:phi}
Assume that the time step is given and satisfies $\dt < \frac{1}{2}$. There exists a function $\phi: \Omega_f \to \mathbb{R}$ satisfying:
\begin{enumerate}
    \item[(i)] $0 \leq  \phi \leq 1$,
    \item[(ii)] $\phi \in V_f$,
    \item[(iii)] $|\{x \in \Sigma : \phi(x) \neq 1 \}| = C\Dt$,
    \item[(iv)] $\|\nabla \phi \|_{L^2(\Omega_f)}^2 \leq C(1+\log{\frac{1}{\Dt}})$,
\end{enumerate}
where each $C$ represents a general constant independent of $\dt$ and the physical parameters.
\end{assumption} 
In the following section, we show how to construct such a $\phi$ in a simple case.

Given the above two assumption we will define the following quantities:
\begin{alignat*}{1}
\tilde{\lla}(x,t):= & \nabla \uu(x,t) \cdot \tilde{\bn}, \\
\Ll(x,t):= & \phi(x) \tilde{\lla}(x,t),\\
\tilde{g}_2^{n+1}:=& \Ll^{n+1}-\Ll^n.
\end{alignat*}
From this we easily see that  $\tilde{g}_2^{n+1} \in V_f$ for all $n$ and can easily prove the following result.
\begin{lemma}
Under Assumptions \ref{ntilde} and \ref{def:phi} we have
\begin{equation}
\|\tilde{g}_2^{n+1}-g_2^{n+1}\|_{L^2(\Sigma)}^2 \le C \dt \|g_2^{n+1}\|_{L^\infty(\Sigma)}^2. \label{lemmag2tilde}
\end{equation} 
\end{lemma}
\begin{proof}
If we use Assumptions \ref{ntilde} and the defintion of $\tilde{g}_2^{n+1}$ we obtain
\begin{alignat*}{1}
\|\tilde{g}_2^{n+1}-g_2^{n+1}\|_{L^2(\Sigma)}^2=&\|(1-\phi)g_2^{n+1}\|_{L^2(\Sigma)}^2 \\
 \le &  \|g_2^{n+1}\|_{L^\infty(\Sigma)}^2 \|1-\phi\|_{L^2(\Sigma)}^2 \\
 \le & C \dt  \|g_2^{n+1}\|_{L^\infty(\Sigma)}^2,
\end{alignat*}
where in the last inequality we used assumption (i) and (iii) of Assumption  \ref{def:phi}.
\end{proof}
We will also need the definition of the second-order difference operator:
\begin{equation*}
\pdt^2 v^n= \frac{v^{n+1}-2v ^n+v^{n-1}}{(\dt)^2}.
\end{equation*}

\begin{lemma}\label{lemmastep3}
Under Assumptions \ref{ntilde}, \ref{def:phi} we have
\begin{equation*}
\frac{\dt}{\alpha} \sum_{n=0}^{M-1} \bl g_3^{n+1}, \Lambda^{n+1} \br \le  \frac{1}{4}  \sum_{n=0}^{M-1} \Ss^{n+1} + \frac{1}{4} \max_{1\le n \le M}  \Z^n+ C \Psi(M).
\end{equation*}
where 
\begin{alignat*}{1}
\Psi(M):=& \Psi_1(M)+ \Psi_2(M),\\
\Psi_1(M):= & \dt \sum_{n=0}^{M-1}   \Bigg( (1+\nu_f)  \| g_1^{n+1}\|_{H^1(\Omega_f)}^2+  \| h_2^{n+1}\|_{L^2(\Omega_f)}^2 
\Bigg) \\
&+  T(\dt)^3 \sum_{n=1}^{M-1} \| \pdt^2 \uu^n\|_{L^2(\Omega_f)}^2+ \|g_1^{M}\|_{L^2(\Omega_f)}^2, \\
\Psi_2(M):=&\dt \sum_{n=0}^{M-1}   \Bigg( \bigg(\frac{1+\nu_f}{\alpha^2} \bigg) \|  \tilde{g}_2^{n+1}\|_{H^1(\Omega_f)}^2+    \| h_2^{n+1}\|_{L^2(\Omega_f)}^2 \Bigg),\\
&+  \frac{ T(\dt)^3}{\alpha^2} \sum_{n=1}^{M-1} \| \pdt^2 \Ll^n\|_{L^2(\Omega_f)}^2+ \frac{1}{\alpha^2} \|\tilde{g}_2^{M}\|_{L^2(\Omega_f)}^2+ \frac{ T \dt}{\alpha} \sum_{n=0}^{M-1}   \|g_2^{n+1}\|_{L^\infty(\Sigma)}^2. 
\end{alignat*}

\end{lemma}
\begin{proof}
Using the definition of $g_3^{n+1}$ we write
\begin{equation*}
\frac{\dt}{\alpha} \sum_{n=0}^{M-1} \bl g_3^{n+1}, \Lambda^{n+1} \br= \frac{\dt}{\alpha} \sum_{n=0}^{M-1} \bl g_2^{n+1}, \Lambda^{n+1} \br + \frac{\dt(k-1)}{2}  \sum_{n=0}^{M-1} \bl g_1^{n+1}, \Lambda^{n+1} \br.
\end{equation*}
We bound the first terms which is the most difficult one to handle.  To this end, 
\begin{equation*}
\frac{\dt}{\alpha} \sum_{n=0}^{M-1} \bl g_2^{n+1}, \Lambda^{n+1} \br= \frac{\dt}{\alpha} \sum_{n=0}^{M-1} \bl g_2^{n+1}- \tilde{g}_2^{n+1} , \Lambda^{n+1} \br+ \frac{\dt}{\alpha} \sum_{n=0}^{M-1} \bl \tilde{g}_2^{n+1}, \Lambda^{n+1} \br.
\end{equation*}
We then see if  we use \eqref{lemmag2tilde} that
\begin{alignat*}{1}
\frac{\dt}{\alpha} \sum_{n=0}^{M-1} \bl g_2^{n+1}- \tilde{g}_2^{n+1} , \Lambda^{n+1} \br \le&  \frac{CT}{\alpha} \sum_{n=0}^{M-1} \| g_2^{n+1}- \tilde{g}_2^{n+1}\|_{L^2(\Sigma)}^2  +  \frac{(\dt)^2}{32T\alpha} \sum_{n=0}^{M-1} \|\Lambda^{n+1}\|_{L^2(\Sigma)}^2 \\
\le &  \frac{C T \dt}{\alpha} \sum_{n=0}^{M-1}   \|g_2^{n+1}\|_{L^\infty(\Sigma)}^2  +    \frac{(\dt)^2}{32T\alpha} \sum_{n=0}^{M-1} \|\Lambda^{n+1}\|_{L^2(\Sigma)}^2.
\end{alignat*}
To estimate the second term we use \eqref{errorRR2} to bound
 \begin{alignat*}{1}
 \dt \bl \tilde{g}_2^{n+1}, \Lambda^{n+1} \br=
 (U^{n+1}-U^n, \tilde{g}_2^{n+1} )_f +\dt \nu_f (\nabla U^{n+1}, \nabla \tilde{g}_2^{n+1})_f -\dt (h_2^{n+1}, \tilde{g}_2^{n+1})_f . 
\end{alignat*}
Thus, we have 
\begin{alignat*}{1}
 \frac{\dt}{\alpha} \sum_{n=0}^{M-1} \bl \tilde{g}_2^{n+1}, \Lambda^{n+1} \br \le &  \frac{1}{\alpha} \sum_{n=0}^{M-1}  (U^{n+1}-U^n, \tilde{g}_2^{n+1} )_f + \frac{\dt \nu_f}{8}  \sum_{n=0}^{M-1} \|\nabla U^{n+1}\|_{L^2(\Omega_f)}^2 \\
 &+C  \dt \sum_{n=0}^{M-1}   \Bigg( \bigg(\frac{1+\nu_f}{\alpha^2} \bigg) \| \tilde{g}_2^{n+1}\|_{H^1(\Omega_f)}^2+  \| h_2^{n+1}\|_{L^2(\Omega_f)}^2 \Bigg).
\end{alignat*}
After using a summation by parts formula and using that $U^0=0$, we get
\begin{alignat*}{1}
&  \frac{1}{\alpha} \sum_{n=0}^{M-1}  (U^{n+1}-U^n, \tilde{g}_2^{n+1} )_f\\
 =&\frac{1}{\alpha} \sum_{n=1}^{M-1}  (U^n, \tilde{g}_2^{n}-\tilde{g}_2^{n+1})_f+\frac{1}{\alpha} (U^M, \tilde{g}_2^{M})_f- \frac{1}{\alpha} (U^0, \tilde{g}_2^{1})_f \\
=& -\frac{(\dt)^2}{ \alpha} \sum_{n=1}^{M-1}  (U^n, \pdt^2 \Ll^{n})_f+\frac{1}{\alpha} (U^M, \tilde{g}_2^{M})_f \\
\le & \frac{\dt}{32 T} \sum_{n=0}^{M-1} \|U^n\|_{L^2(\Omega_f)}^2+  \frac{C T(\dt)^3}{\alpha^2} \sum_{n=1}^{M-1} \| \pdt^2 \Ll^n\|_{L^2(\Omega_f)}^2 +  \frac{1}{32} \|U^M\|_{L^2(\Omega_f)}^2+\frac{C}{\alpha^2} \|\tilde{g}_2^{M}\|_{L^2(\Omega_f)}^2.
\end{alignat*}
Combining the above inequalities  we get 
\begin{alignat*}{1}
\frac{\dt}{\alpha} \sum_{n=0}^{M-1} \bl g_2^{n+1}, \Lambda^{n+1} \br \le &   \frac{\dt \nu_f}{8} \sum_{n=0}^{M-1} \|\nabla U^{n+1}\|_{L^2(\Omega_f)}^2+\frac{\dt}{16 T} \sum_{n=1}^{M} \Big( \frac{1}{2}\|U^n\|_{L^2(\Omega_f)}^2+ \frac{\dt}{2\alpha} \|\Lambda^n\|_{L^2(\Sigma)}^2\Big)\\
&+ \frac{1}{32} \|U^M\|_{L^2(\Omega_f)}^2+ C \Psi_2(M) \\
\le &  \frac{1}{8} \sum_{n=0}^{M-1} \Ss^{n+1}+  \frac{1}{8} \max_{1\le n \le M}  \Z^n+ C \Psi_2(M).
\end{alignat*}
To bound  $\frac{\dt(k-1)}{2}  \sum_{n=0}^{M-1} \bl g_1^{n+1}, \Lambda^{n+1} \br$ is much easier since $g_1^{n+1}$ already belongs to $V_f$ and we do not have to use a cut-off function technique.  Using similar arguments as bounding the other term we can prove 
\begin{alignat*}{1}
\frac{\dt(k-1)}{2}  \sum_{n=0}^{M-1} \bl g_1^{n+1}, \Lambda^{n+1} \br \le &    \frac{(k-1)}{8}  \sum_{n=0}^{M-1} \Ss^{n+1}\\
&+\frac{(k-1)}{8} \max_{1\le n \le M}  \Z^n+ C(k-1) \Psi_1(M).
\end{alignat*}

\end{proof}

We would like to mention that there are a few special cases where we can relax (ii) of Assumption \ref{def:phi}. The first case is when  $\Sigma$ is straight line and $\Omega_f$ is a rectangle with two sides perpendicular to $\Sigma$. In this case $\nabla \uu \cdot n$ vanishes on those two sides so we do not have to make $\phi$ to vanish there. Then, one can construct $\phi$ so (iv) can be improved: $\|\nabla \phi\|_{L^2(\Omega_f)} \le C$. 
This will give  an improved  estimate of $O(\dt)$ for the final theorem below instead of $O\bigg( \dt \sqrt{T + \log(\frac{1}{\dt})}\bigg)$. The other case is when $\Sigma$ does not touch the boundary of $\Omega$.

\subsection{Proof of the main result}
In this section we  put the above estimates together to prove our  main result.  In order to do this, we need to estimate $D(M)$ and $\Psi(M)$.  
We will use the Bochner norms $\|v\|_{L^2(a,b; X)}=\Big(\int_a^b \|v(\cdot, s)\|_X^2 ds\Big)^{1/2}$ and $\|v\|_{L^\infty(a,b; X)}=\text{ess sup}_{ a \le  s \le b}  \|v(\cdot, s)\|_X$.  

For  $X$ a Sobolev space, it is well known that
\begin{subequations}
\begin{alignat}{1}\label{721}
\| v^{n+1}-v^n\|_X^2 \le&  C \dt \int_{t_n}^{t_{n+1}} \|\pt v(\cdot,s)\|_{X}^2 ds, \\
\| \pdt v^{n+1}-\pt v^{n+1}\|_X^2 \le&  C \dt \int_{t_n}^{t_{n+1}} \|\pt^2 v(\cdot,s)\|_{X}^2 ds, \\
\|\pdt v^{n+1}-\pt v^{n+1/2}\|_X \leq& C(\dt)^3 \int_{t_n}^{t_{n+1}} \|\pt^3 v(\cdot,s)\|_X ds,\\
\int_{a}^{b} \|v(\cdot,s)\|_X^2 ds \leq & (b-a) \|v\|_{L^\infty(a,b;X)}^2.
\end{alignat}
\end{subequations}
The following identity can easily be shown
\begin{equation*}
 \pdt^2 v^n= \frac{1}{(\dt)^2} \int_{-\dt}^{\dt} \pt^2 v(\cdot, t_n+s) ds.
\end{equation*}
From this we can show that
\begin{equation}\label{621}
\| \pdt^2 v^n\|_{L^2(\Omega_i)}^2 \le  \frac{C}{\dt} \int_{t_{n-1}}^{t_{n+1}} \|\pt^2 v(\cdot,s)\|_{L^2(\Omega_i)}^2 ds, \quad i=s,f.
\end{equation}
Using \eqref{721} and \eqref{621} we can prove the following approximation lemma. 
\begin{lemma}\label{DPSI}
It holds,
\begin{alignat*}{1}
D(N) \le&  CT(\dt)^{2k}\Big(\|\pt^{k+1} \qq\|_{L^2(0,T;L^2(\Omega_s))}^2  + (k-1) \nu_s \|\pt^{k+1} \nabla \ww\|_{L^2(0,T;L^2(\Omega_s))}^2\Big)\\
& \hspace{1.5cm}+C T (\dt)^2 \|\pt^2 \uu\|_{L^2(0,T;L^2(\Omega_f))}^2 +(\dt)^2 \Big( ( \alpha+ \frac{\alpha^2}{\nu_f}) \|\pt \uu\|_{L^2(0,T;L^2(\Sigma))}^2 \\
&  \hspace{1.5cm} + ( \frac{1}{\alpha}+ \frac{1}{\nu_f}) \|\pt \lla\|_{L^2(0,T;L^2(\Sigma))}^2 \Big), \\
\max_{1 \le M \le N} \Psi_1(M) \le &  C (\dt)^2 \Bigg((1+\nu_f)\|\pt \uu\|_{L^2(0,T;H^1(\Omega_f))}^2 + (1+T) \|\pt^2 \uu\|_{L^2(0,T;L^2(\Omega_f))}^2 \\
&\hspace{1.5cm}+  \|\pt \uu\|_{L^\infty(0,T;L^2(\Omega_f))}^2 \Big), \\
\max_{1 \le M \le N} \Psi_2(M) \le &  C (\dt)^2  \Bigg( \frac{\nu_f^2(1+\nu_f)}{\alpha^2} \bigg( \|\pt \uu\|_{L^2(0,T;H^2(\Omega_f))}^2+ (1+ \log(\frac{1}{\dt}))   \|\nabla \pt \uu\|_{L^2(0,T; L^\infty(\Omega_f))}^2\bigg) \\
 &  \hspace{1.5cm}+ \bigg(1+\frac{ \nu_f^2 T}{\alpha^2}\bigg)  \|\pt^2 \uu\|_{L^2(0,T;H^1(\Omega_f))}^2  +\frac{\nu_f^2}{\alpha^2} \| \nabla \pt \uu\|_{L^\infty(0,T; L^2(\Omega_f))}^2  \\
 &\hspace{1.5cm}+ \frac{T}{\alpha}  \|\pt \lla\|_{L^2(0,T;L^\infty(\Sigma))}^2   \Bigg).
\end{alignat*}
\end{lemma}
\begin{proof}
The estimate for $D(N)$ and $\Psi_1(M)$ are  straightforward application of \eqref{721} and \eqref{621}.  For $\Psi_2(M)$  let us bound the most difficult term. Using the definition of $\tilde{g}_2^{n+1}$, the product rule, Assumption \ref{ntilde} and (i) and (iv) of Assumption \ref{def:phi} we obtain  
\begin{alignat*}{1}
& \dt\bigg(\frac{1+\nu_f}{\alpha^2} \bigg) \sum_{n=0}^{M-1}   \|  \nabla \tilde{g}_2^{n+1}\|_{L^2(\Omega_f)}^2 \\
\le & \dt\bigg(\frac{\nu_f^2(1+\nu_f)}{\alpha^2} \bigg) \sum_{n=0}^{M-1}  \Big(\| \nabla \phi \|_{L^2(\Omega_f)}^2 \|\nabla(\uu^{n+1}-u^n)\|_{L^\infty(\Omega_f)}^2 + \|\nabla \tilde{\bn}\|_{L^\infty(\Omega_f)}^2\| \nabla(\uu^{n+1}-u^n)\|_{L^2(\Omega_f)}^2 \\
&  \hspace{3.8cm}+\|D^2 (\uu^{n+1}-\uu^n)\|_{L^2(\Omega_f)}^2 \Big) \\ 
\le &\dt\bigg(\frac{\nu_f^2(1+\nu_f)}{\alpha^2} \bigg) \sum_{n=0}^{M-1}  \Big( \big(1+\log(\frac{1}{\dt}\big)) \|\nabla(\uu^{n+1}-\uu^n)\|_{L^\infty(\Omega_f)}^2 + \|D^2 (\uu^{n+1}-\uu^n)\|_{L^2(\Omega_f)}^2 \Big).
\end{alignat*}
Here $D^2 u$ denotes the Hessian of $u$. 

Using \eqref{721} we get 
\begin{alignat*}{1}
\dt\bigg(\frac{1+\nu_f}{\alpha} \bigg) \sum_{n=0}^{M-1}   \|  \nabla \tilde{g}_2^{n+1}\|_{L^2(\Omega_f)}^2 
 \le & C  (\dt)^2\bigg(\frac{\nu_f^2(1+\nu_f)}{\alpha^2} \bigg) \Bigg(  \|\pt \uu\|_{L^2(0,T;H^2(\Omega_f))}^2\\
 &\hspace{4cm}+ (1+ \log(\frac{1}{\dt}))   \|\nabla \pt \uu\|_{L^2(0,T; L^\infty(\Omega_f))}^2 \Bigg).
\end{alignat*}
We leave the bounds of the remaining terms to the reader. These follow from \eqref{721} and \eqref{621}.   
\end{proof}

Now we can state and prove the main result of the paper. 
\begin{theorem}Under Assumptions \ref{ntilde}, \ref{def:phi}, and assuming that the solution is smooth enough so that   $\mathsf{Y}$ defined below and $\|\nabla \pt \uu\|_{L^2(0,T; L^\infty(\Omega_f))}^2$ are bounded we have
\begin{alignat*}{1}
\max_{1 \le M \le N} \Z^{M}+ \sum_{n=0}^{N-1} \Ss^{n+1} \le    C (\dt)^2 \mathsf{Y}+ \frac{C \nu_f^2(1+\nu_f) (\dt)^2}{\alpha^2} \big(1+\log(\frac{1}{\dt})\big)  \|\nabla \pt \uu\|_{L^2(0,T; L^\infty(\Omega_f))}^2,
\end{alignat*}
where
\begin{alignat*}{1}
\mathsf{Y}:=& (\dt)^{2k-2}\bigg(\|\pt^{k+1}\qq\|_{L^2(0,T;L^2(\Omega_s))}^2 + (k-1)\nu_s \|\pt^{k+1}\nabla \ww\|_{L^2(0,T;L^2(\Omega_s))}^2\bigg) \\
&+\bigg(1+\big(1+\frac{\nu_f^2}{\alpha^2}\big)T\bigg)\|\pt^2\uu\|_{L^2(0,T;H^1(\Omega_f))}^2 + \big(\alpha + \frac{\alpha^2}{\nu_f}\big)\|\pt \uu\|_{L^2(0,T;L^2(\Sigma))}^2 \\
&+ \big( \frac{1}{\alpha}+\frac{1}{\nu_f}\big) \|\pt \lla \|_{L^2(0,T;L^2(\Sigma))}^2 + (1+\nu_f)\big(1+ \frac{\nu_f^2}{\alpha^2}\big)\|\pt \uu \|_{L^2(0,T;H^2(\Omega_f))}^2 \\
&+ \big(1+\frac{\nu_f^2}{\alpha^2}\big)\|\pt \uu \|_{L^\infty(0,T;H^1(\Omega_f))}^2 +\frac{T}{\alpha}\|\pt \lla\|_{L^2(0,T;L^\infty(\Sigma))}^2.
\end{alignat*}
\end{theorem}

\begin{proof}
Taking the sum in Lemma \ref{errorLemma1} and using Lemmas \ref{Festimate} , \ref{lemmastep3} 
we have  for all $1 \le M \le N$
\begin{alignat*}{1}
\Z^{M}+ \frac{1}{2} \sum_{n=0}^{M-1} \Ss^{n+1} \le    \frac{1}{2} \max_{1\le n \le M}  \Z^n + C(\Psi(M)+ D(M)).
\end{alignat*}
Taking the maximum over $1 \le M \le N$ we get 
\begin{alignat*}{1}
\frac{1}{2} \max_{1 \le M \le N} \Z^{M}\le    C\max_{1 \le M \le N} (\Psi(M)+ D(M)) .
\end{alignat*}
Hence, we also get
\begin{alignat*}{1}
\sum_{n=0}^{M-1} \Ss^{n+1}\le    C\max_{1 \le M \le N} (\Psi(M)+ D(M)) .
\end{alignat*}
The result now follows after using Lemma \ref{DPSI}.
\end{proof} 

\section{Construction of $\phi$}\label{sec:phi}
In this section, we will provide a construction of $\phi$ when $\Sigma$ is a straight line and $\Omega_f$ is a  rectangle with two sides perpendicular to $\Sigma$.  The construction carries over to the case when $\Sigma$ is  a straight line and the domain $\Omega_f$ is a polygon, but for simplicity we consider the simplified case only. We anticipate that the construction holds in more general setting when $\Sigma$ and $\Omega_f \backslash \Sigma$ are sufficiently smooth.

Let $\Omega_f=[0,1] \times [0, 1]$ which is illustrated in  Figure \ref{fig:simpleOmegf}. Note that $\Sigma=[0,1] \times \{1\}$ marks the top boundary of $\Omega_f$. We begin the construction by dividing $\Omega_f$ into five regions, $\{K_i\}_{i=1}^5$ which are depicted in Figure \ref{fig:simpleOmegf}. We may then define $\phi(x)=\phi(x_1,x_2)$ piece-wise as follows:
\begin{equation}\label{simplePhi}
\phi(x):=\begin{cases}
\frac{x_1}{1-(1-\Dt)x_2}, & x \in K_1, \\
1, & x \in K_2, \\
\frac{1-x_1}{1-(1-\Dt)x_2}, & x \in K_3, \\
4x_2 (1-x_1)(1-\Dt), & x \in K_4, \\
4x_1 x_2 (1-\Dt), & x \in K_5.
\end{cases}
\end{equation}

\begin{figure}[h]
\includegraphics[width=0.7\textwidth]{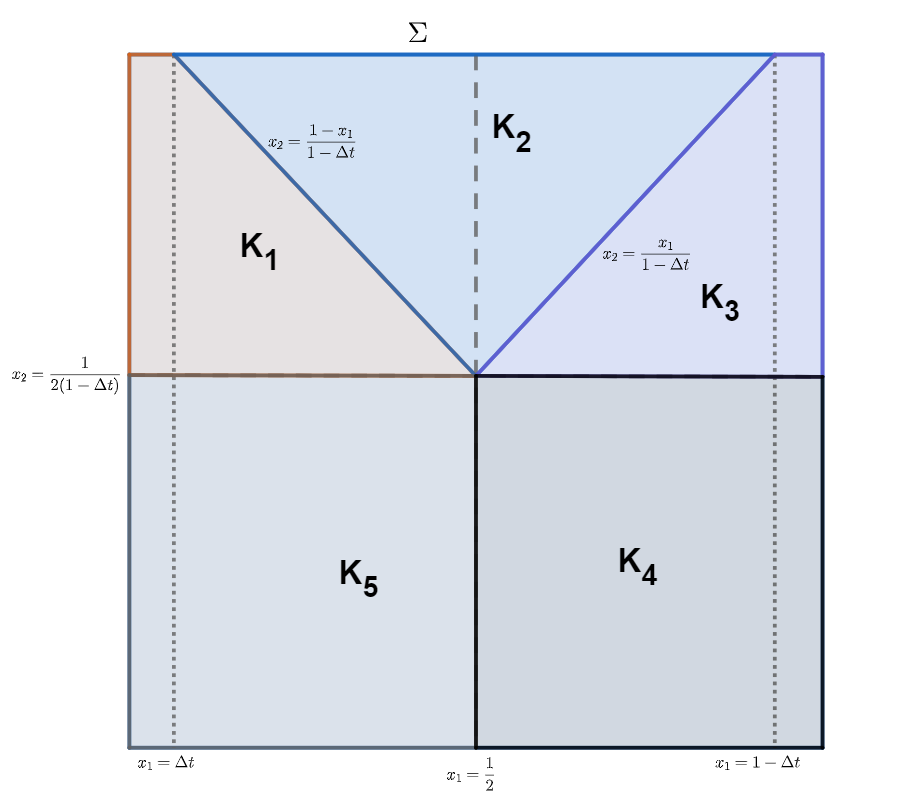}
\caption{$\Omega_f = [0,1]^2$ with the regions $\{K_i\}_{i=1}^5$.}\label{fig:simpleOmegf}
\end{figure}

It is straightforward to check that this $\phi$ is piecewise-continuous and satisfies conditions (i)-(iii) in Definition \ref{def:phi}. For condition (iv), we note that
\begin{equation*}
    \int_{\Omega_f}|\nabla \phi|^2 = \big(\frac{2}{1-\Dt} + \frac{2(1-\Dt)}{3}\big) \log{\frac{1}{2\Dt}} + \frac{2(1-\Dt)}{3} + \frac{2}{3(1-\Dt)}.
\end{equation*}
Thus, if $\Dt < \frac{1}{2}$, we have
\begin{equation*}
    \int_{\Omega_f}|\nabla \phi|^2 \leq C(1+\log\frac{1}{\Dt})
\end{equation*}

\section{Numerical Experiments}\label{num}

In this section, we seek to verify our theoretical results through numerical experiments. As our analysis above is carried out in a time-semi-discrete framework, the spacial discretization may be done with any appropriate numerical method. For simplicity, we chose to use linear finite elements with mesh parameters defined below for each case. To accommodate our Lagrange multiplier method, we require that the meshes on $\Omega_s$ and $\Omega_f$ match at the interface $\Sigma$. In all cases below, we set the coefficients $\nu_f = \nu_s = \alpha = 1$. 
\section{Numerical experiments for simple case with uniform mesh}

 For our first case, We take $\Omega=[0,1]^2$ and $\Omega_f=[0,1] \times [0,3/4]$ and $\Omega_s=[0,1] \times [3/4, 1]$ and of course $\Sigma=\overline{\Omega_f} \cap \overline{\Omega_s}$. On each subdomain, we use a uniform mesh with $h=\Dt$. We run the system to time $T=0.25$.

For the Parabolic-Parabolic ($k=1$) coupling, we take the exact solution to be
\begin{alignat*}{1}
u(t,x) &=e^{-2 \pi^2t } \sin(\pi x_1) \sin(\pi x_2), \\
w(t,x) &=e^{-2 \pi t} \sin(\pi x_1) \sin(\pi x_2),\\
\lla(t,x) &= \pi e^{-2 \pi t} \sin(\pi x_1)\cos(\pi x_2).
\end{alignat*}

For the Parabolic-Hyperbolic ($k=2$) coupling, we take the exact solution to be
\begin{alignat*}{1}
u(t,x)&=(10^{-3})e^t x_1(1-x_1)x_2(1-x_2), \\
w(t,x)&=(10^{-3})e^t x_1(1-x_1)x_2(1-x_2), \\
\lla(t,x)&=(10^{-3})e^t x_1(1-x_1)(1-2x_2).
\end{alignat*}

In the table below, we measure the error in the $L^2$-norm at the final time $T=0.25$.
 
\begin{figure}[h]
\begin{center}

\begin{tabular}{|c||c|c|c|c||c|c|c|c|}
\hline
{} & \multicolumn{4}{|c||}{$k=1$} & \multicolumn{4}{|c|}{$k=2$} \\
\hline
  {} & \multicolumn{2}{|c|}{Rates} &  \multicolumn{2}{|c||}{Errors} &  \multicolumn{2}{|c|}{Rates} &  \multicolumn{2}{|c|}{Errors} \\
\hline
$\Dt$ &  $U$ &  $W$ &  $U$ & $W$ & $U$ &  $Q$ & $U$ & $Q$  \\
\hline
$(\frac{1}{2})^2$ & -- & -- 
&0.04 & 0.09 & -- & --
& 4.48e-06 & 9.48e-06
\\
\hline
$(\frac{1}{2})^3$&-0.30 &0.59 & 0.05 & 0.06 & 2.16 & 1.94 & 1.01e-06 & 2.45e-06
\\
\hline
$(\frac{1}{2})^4$ &0.22 &  0.93& 0.04 & 0.03 & 0.96 & 1.02 & 5.19e-07 & 1.22e-06
\\
\hline
$(\frac{1}{2})^5$ &1.08 & 1.67&0.02 & 9.6e-03 & 0.47 & 0.89 & 3.76e-07 & 6.67e-07 
\\
\hline
$(\frac{1}{2})^6$ &2.43 & 3.71 & 3.5e-03 & 7.3e-04 & 0.74 & 0.88 & 2.25e-07 & 3.62e-07
\\
\hline 
$(\frac{1}{2})^7$ &2.89 & 1.17 &
 4.8e-04 & 3.3e-04 & 0.88 & 0.93 & 1.22e-07 & 1.89e-07
 \\
\hline
$(\frac{1}{2})^8$ &1.77 & 0.85 & 1.4e-04  & 1.8e-04 & 0.94 & 0.97 & 6.36e-08 & 9.69e-08
\\
\hline
$(\frac{1}{2})^{9}$ &1.37 & 0.97 & 5.4e-05 & 9.6e-05  & 0.97 & 0.98 & 3.24e-08 & 4.91e-08
\\
\hline

\end{tabular}
\end{center}
\caption{Convergence rates and errors for the parabolic-parabolic problem in the special case where $\Sigma$ is perpendicular to $\partial \Omega$.}\label{fig:conv1}
\end{figure}
Clearly, we can see that all the error terms behave as $O(\Dt)$ as $\Dt \to 0$, which agrees well with our theoretical results. 

 \subsection{Numerical experiments for less simple case with non-uniform mesh}

For our next set of numerical experiments, we investigate the case where $\Sigma$ is not perpendicular to $\partial \Omega$. We again take $\Omega=[0,1]^2$, however $\Sigma$ is now represented by the line $x_2 = \frac{x_1}{2} + \frac{1}{4}$. Furthermore, we generate a non-uniform mesh such that the interface $\Sigma$ is aligned with the nodes of the mesh (See Figure \ref{fig:mesh1}). We run our method on this grid for $\Delta t = \{(\frac{1}{2})^{n+2}\}_{n=0}^{8}$. For each $n$, the maximum element diameter, $h_{max}$, is given (to 4 significant digits) in the list below. From this we can see that $h$ is roughly $O(\Delta t)$. Additionally, as before, we run to time $T=0.25$.

\begin{alignat*}{1}
(n, h_{max}) = \{&(0,0.3125), \ (1,0.1574), \ (2,0.0794), \ (3,0.0398), \ (4,0.0199), \ (5,0.0099), \\
&(6,0.0050), \ (7,0.0025), \ (8,0.0012) \}.
\end{alignat*}

\begin{figure}[h]
\includegraphics[width=0.5\textwidth]{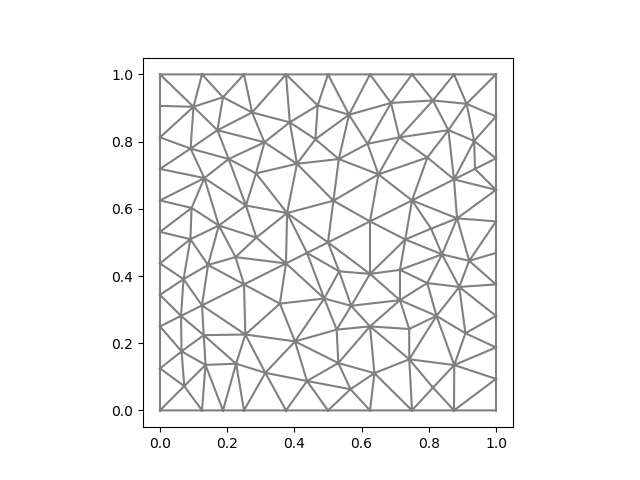}
\caption{An example of the non-uniform mesh for $n=1$. \label{fig:mesh1}}
\end{figure}

We take the exact solution to be
\begin{alignat*}{1}
\uu(t,x)&=(10^{-3})e^t x_1(1-x_1)x_2(1-x_2) \\
\ww(t,x)&=(10^{-3})e^t x_1(1-x_1)x_2(1-x_2) \\
\lla(t,x)&=\left(\frac{10^{-3}}{\sqrt{5}}\right)e^t \left( 2x_1(1-x_1)(1-2x_2) - (1-2x_1)x_2(1-x_2) \right),
\end{alignat*}
and we note that $\lla = \nabla \uu \cdot \bn$ no longer vanishes on $\partial \Omega$.

In the table below, we again measure the $L^2$-norm of $U$ and $W$ at the final time $T=0.25$. However, we also measure $\big(\dt \sum_{n=1}^{N} \|\nabla U^n\|_{L^2(\Omega_f)}^2 \big)^{1/2}$ and $\big(\dt \sum_{n=1}^{N} \|\nabla W^n\|_{L^2(\Omega_s)}^2 \big)^{1/2}$, where $N=\frac{T}{\dt}$.

\begin{figure}[h]
\begin{center}

\begin{tabular}{|c||c|c|c|c|c|c|c|c|}
\hline
  {} & \multicolumn{4}{|c|}{Rates} &  \multicolumn{4}{|c|}{Errors}
\\
\hline
$\Dt$ &  Rate $U$ & Rate $W$ & Rate $\nabla U$ & Rate $\nabla W$&  $U$ & $W$ & $\nabla U$ & $\nabla W$
\\
\hline
$(\frac{1}{2})^2$ & -- & -- & -- & -- & 4.92e-07 & 1.27e-06 &
{4.60e-06 } &{5.28e-06 }
\\
\hline
$(\frac{1}{2})^3$&
 1.28 & 1.87 
& {0.99} & {1.45} & 2.02e-07 & 3.48e-07 &
2.33e-06 & 1.93e-06 \\
\hline
$(\frac{1}{2})^4$ & 0.59 & 0.83
& {1.12} & {0.78} & 1.35e-07 & 1.95e-07 & 1.07e-06 & 1.13e-06 
\\
\hline
$(\frac{1}{2})^5$ & 0.45 & 0.87 
& {0.77} & {0.81} & 9.85e-08 & 1.07e-07 & 6.26e-07 & 6.43e-07 
\\
\hline
$(\frac{1}{2})^6$ & 0.71 & 0.81
& {0.66} & {0.69} & 6.01e-08 & 6.09e-08 & 3.96e-07  & 4.00e-07 
\\
\hline 
$(\frac{1}{2})^7$ & 0.89 & 0.89
& {0.69} & {0.69} & 3.26e-08 & 3.28e-08 & 2.45e-07 & 2.47e-07
\\
\hline
$(\frac{1}{2})^8$ & 0.94 & 0.95
&{0.74} & {0.75} & 1.69e-08 & 1.70e-08 & 1.46e-07 & 1.47e-07 
\\
\hline
$(\frac{1}{2})^{9}$ & 0.97 & 0.98
& {0.79} & {0.80} & 8.62e-09 & 8.63e-09 & 8.44e-08 & 8.45e-08
\\
\hline
$(\frac{1}{2})^{10}$ & 0.99 &  0.99 
&{0.83} & {0.83} & 4.35e-09 & 4.36e-09 & 4.74e-08 & 4.74e-08
\\
\hline
\end{tabular}
\end{center}
\caption{Convergence rates and errors for the parabolic-parabolic problem for a more general case.}
\label{fig:conv1}
\end{figure}

As before, we see first order convergence in $U$ and $W$. For $\nabla U$ and $\nabla W$, the convergence is slower, however it seems it will eventually approach first order. 

\section{Conclusion}
We analyzed the Robin-Robin splitting method for Parabolic/Parabolic and Parabolic/Hyperbolic coupled systems in a unified, time semi-discrete framework. This leaves many options for spatial discretization, and we present numerical experiments using linear finite elements.  

Of particular interest to us is the potential extension of these results to the FSI problem. However, in this case, the presence of the pressure term in $\Omega_f$ will mean the extension from the interface to the interior must be handled with care. We hope to further explore this problem in a future paper.

\bibliographystyle{abbrv}
\bibliography{papperPPandWP}

\begin{thebibliography}{10}

\bibitem{badia2008fluid}
S.~Badia, F.~Nobile, and C.~Vergara.
\newblock Fluid--structure partitioned procedures based on {Robin} transmission
  conditions.
\newblock {\em Journal of Computational Physics}, 227(14):7027--7051, 2008.

\bibitem{BHS14}
J.~W. Banks, W.~D. Henshaw, and D.~W. Schwendeman.
\newblock An analysis of a new stable partitioned algorithm for {FSI} problems.
  {P}art {I}: {I}ncompressible flow and elastic solids.
\newblock {\em J. Comput. Phys.}, 269:108--137, 2014.

\bibitem{benes1}
M.~Benes.
\newblock Convergence and stability analysis of heterogeneous time step
  coupling schemes for parabolic problems.
\newblock {\em Applied Mathematics and Computation}, 121:198--222, 2017.

\bibitem{benes2}
M.~Benes, A.~Nekvinda, and M.~K. Yadav.
\newblock Multi-time-step domain decomposition method with non-matching grids
  for parabolic problems.
\newblock {\em Applied Mathematics and Computation}, 267:571--582, 2015.

\bibitem{bukavc2021refactorization}
M.~Buka{\v{c}}, A.~Seboldt, and C.~Trenchea.
\newblock Refactorization of cauchy’s method: A second-order partitioned
  method for fluid--thick structure interaction problems.
\newblock {\em Journal of Mathematical Fluid Mechanics}, 23(3):1--25, 2021.

\bibitem{burman-durst-guzman-fernandez}
E.~Burman, R.~Durst, M.~Fernandez, and J.~Guzman.
\newblock Fully discrete loosely coupled {Robin}-{Robin} scheme for
  incompressible fluid-structure interaction: stability and error analysis,
  2020.

\bibitem{burman-durst-guzman-19}
E.~Burman, R.~Durst, and J.~Guzman.
\newblock Stability and error analysis of a splitting method using
  {Robin}-{Robin} coupling applied to a fluid-structure interaction problem,
  2019.

\bibitem{BURMAN2009766}
E.~Burman and M.~A. Fern\'andez.
\newblock Stabilization of explicit coupling in fluid–structure interaction
  involving fluid incompressibility.
\newblock {\em Computer Methods in Applied Mechanics and Engineering},
  198(5):766--784, 2009.

\bibitem{burman2014explicit}
E.~Burman and M.~A. Fern{\'a}ndez.
\newblock Explicit strategies for incompressible fluid-structure interaction
  problems: Nitsche type mortaring versus {Robin}--{Robin} coupling.
\newblock {\em International Journal for Numerical Methods in Engineering},
  97(10):739--758, 2014.

\bibitem{canuto}
C.~Canuto and A.~Lo~Giudice.
\newblock A multi-timestep {Robin}-{Robin} domain decomposition method for time
  dependent advection-diffusion problems.
\newblock {\em Applied Mathematics and Computation}, 363:124596, 2019.

\bibitem{stokesDarcy}
Y.~Cao, M.~Gunzberger, X.~He, and X.~Wang.
\newblock Parallel, non-iterative, multiphysics domain decomposition methods
  for time-dependent {Stokes-Darcy} systems.
\newblock {\em Mathematics of Computation}, 83(288):1617--1644, 2014.

\bibitem{causin2005added}
P.~Causin, J.-F. Gerbeau, and F.~Nobile.
\newblock Added-mass effect in the design of partitioned algorithms for
  fluid--structure problems.
\newblock {\em Computer methods in applied mechanics and engineering},
  194(42-44):4506--4527, 2005.

\bibitem{FWR07}
C.~F\"{o}rster, W.~A. Wall, and E.~Ramm.
\newblock Artificial added mass instabilities in sequential staggered coupling
  of nonlinear structures and incompressible viscous flows.
\newblock {\em Comput. Methods Appl. Mech. Engrg.}, 196(7):1278--1293, 2007.

\bibitem{GNV10}
L.~Gerardo-Giorda, F.~Nobile, and C.~Vergara.
\newblock Analysis and optimization of {R}obin-{R}obin partitioned procedures
  in fluid-structure interaction problems.
\newblock {\em SIAM J. Numer. Anal.}, 48(6):2091--2116, 2010.

\bibitem{le-tallec-mouro-01}
P.~Le~Tallec and J.~Mouro.
\newblock Fluid structure interaction with large structural displacements.
\newblock {\em Comput. Meth. Appl. Mech. Engrg.}, 190:3039--3067, 2001.

\bibitem{lions1990schwarz}
P.-L. Lions.
\newblock On the {Schwarz} alternating method. iii: a variant for
  nonoverlapping subdomains.
\newblock In {\em Third international symposium on domain decomposition methods
  for partial differential equations}, volume~6, pages 202--223. SIAM
  Philadelphia, PA, 1990.

\bibitem{NV08}
F.~Nobile and C.~Vergara.
\newblock An effective fluid-structure interaction formulation for vascular
  dynamics by generalized {R}obin conditions.
\newblock {\em SIAM J. Sci. Comput.}, 30(2):731--763, 2008.

\bibitem{bukacSeboldt}
A.~Seboldt and M.~Buka\v{c}.
\newblock A non-iterative domain decomposition method for the interaction
  between a fluid and a thick structure.
\newblock {\em Numerical Methods for Partial Differential Equations},
  37(4):2803--2832, 2021.

\end{thebibliography}

\end{document}